\documentclass[reqno]{amsart}
\usepackage[margin=2.5cm]{geometry}
\usepackage[utf8]{inputenc}
\usepackage[T1]{fontenc}

\usepackage{amsmath,amsthm,amssymb,amsfonts}
\usepackage{graphicx,subfig}
\usepackage[english]{babel}
\usepackage{xcolor}
\usepackage{hyperref}

\newcommand{\R}{\mathbb{R}}
\newcommand{\N}{\mathbb{N}}

\newcommand{\D}{\mathbb{D}}

\newcommand{\s}{\mathbb{S}}
\newcommand{\eps}{\varepsilon}

\newtheorem{theoreme}{Theorem}[section]
\newtheorem{corollaire}[theoreme]{Corollary}
\newtheorem{proposition}[theoreme]{Proposition}
\newtheorem{definition}[theoreme]{Definition}
\newtheorem{lemme}[theoreme]{Lemma}
\newtheorem{exemple}[theoreme]{Example}
\newtheorem{remarque}[theoreme]{Remark}

\title{Tubular excision and Steklov eigenvalues}
\author{Jade Brisson}
\address{D\'epartement de math\'ematiques et de statistique, Pavillon Alexandre-Vachon, Universit\'e Laval, Qu\'ebec, QC, G1V 0A6, Canada}
\email{jade.brisson.1@ulaval.ca}

\begin{document}
\begin{abstract}
  Given a closed manifold $M$ and a closed connected submanifold $N\subset M$ of positive codimension, we study the Steklov spectrum of the domain $\Omega_\eps\subset M$ obtained by removing the tubular neighbourhood of size $\eps$ around $N$. All non-zero eigenvalues in the mid-frequency range tend to infinity at a rate which depends only on the codimension of $N$ in $M$. Eigenvalues above the mid-frequency range are also described: they tend to infinity following an unbounded sequence of clusters. This construction is then applied to obtain manifolds with unbounded perimeter-normalized spectral gap and to show the necessity of using the injectivity radius in some known isoperimetric-type upper bounds.
\end{abstract}

\maketitle

\section{Introduction}

Let $(\Omega,g)$ be a smooth compact connected Riemannian manifold of dimension $m\geq2$, with boundary $\partial\Omega$.
A real number $\sigma\in\R$ is called a Steklov eigenvalue if there exists a nonzero function $f\in C^\infty(\Omega)$ such that
\begin{equation*}
\begin{cases}
    \Delta f=0 &\mbox{ in }\Omega,\\
    \partial_n f=\sigma f &\mbox{ on }\partial\Omega.
\end{cases}
\end{equation*}
Here and elsewhere $\Delta=\Delta_g:C^\infty(\Omega)\to C^\infty(\Omega)$ is the Laplace operator induced by the Riemannian metric $g$, and $\partial_n$ denotes the outward-pointing normal derivative on $\partial\Omega$.  It is well known that the Steklov eigenvalues of $\Omega$ form a sequence
\[0=\sigma_0(\Omega,g)<\sigma_1(\Omega,g)\leq\sigma_2(\Omega,g)\leq\ldots\nearrow +\infty\,,\]
where each eigenvalue is repeated according to its multiplicity. This sequence is known as the Steklov spectrum of $(\Omega,g)$ and will be denoted $\mathcal{S}(\Omega)$. The question to link the Steklov eigenvalues of the manifold $\Omega$ to its geometry is an active research topic in spectral geometry. In particular, several authors have proved upper bounds for the Steklov eigenvalues under various geometric constraints. See~\cite{girouard2020large,colbois2020upper,colbois2020sharp, HenrotDiam}. In parallel, it is interesting to construct various examples of manifolds that have large first nonzero Steklov eigenvalue $\sigma_1$, as this can be used to study the relevance of various upper bounds. See~\cite{CiG,CGR,ColbGirGraphSurface}. The current paper provides a novel way to obtain perimeter-normalized manifolds with large spectral gap $\sigma_1>0$, with a particularly simple geometry that is obtained by removing thin tubular neighborhoods of closed manifolds of positive codimension. 

\subsection{Tubular excision of closed Riemannian manifolds}
Let $M$ be a smooth compact Riemannian manifold of dimension $m$, without boundary. Given a closed submanifold $N\subset M$ of positive codimension $m-n$, consider the tubular neighbourhoods $T_\eps=\{x\in M\,:\,d(x,N)<\eps\}$, where $d=d_g$ is the Riemannian distance. We study the Steklov eigenvalues of the domains
\begin{gather}\label{def:perforateDom}
  \Omega_\eps:=M\setminus T_\eps=\{x\in M\,:\,d(x,N)\geq\eps\},
\end{gather}
obtained by excision of the tubular neighbourhood $T_\eps$.
The main result of this paper is a description of the asymptotic  behaviour of the Steklov eigenvalues of these domains as $\eps\to 0$. 
This description involves the Laplace eigenvalues of the unit $d-$dimensional sphere, for $d=m-n-1$. The distinct Laplace eigenvalues $\{\mu_{(j)}\}$ are given by $\mu_{(j)}=j(j+d-1)$. The multiplicity of $\mu_{(j)}$ is $m_j=\binom{d+j}{d}-\binom{d+j-2}{d}$.

\begin{theoreme}\label{thm:PrincipalIntro}
Let $M$ be a compact Riemannian manifold of dimension $m\geq 3$ and let $N\subset M$ be a closed connected submanifold of dimension $0<n\leq m-2$.
Then, for all $k\,,\ell\geq 0$, there are numbers $\sigma_{k,\ell}(\eps)\geq 0$ such that the Steklov spectrum of $\Omega_\eps$ is given by the multi-set
$$\mathcal{S}(\eps)=\{\sigma_{k,\ell}(\eps)\,:\,\ell,k\geq 0\},$$
where $\sigma_0(\Omega_\eps)=\sigma_{0,0}(\eps)\equiv 0$. For $\ell=0$, set $j=0$ and for $\ell>0$, choose the unique $j>0$ such that $m_0+\cdots+m_{j-1}\leq \ell< m_0+\cdots+m_{j-1}+m_j$. In each of these cases the following limits hold, for each $k\geq0$,
\begin{gather}\label{eq:mainresultIntro}
  \lim\limits_{\varepsilon\to0}\varepsilon\sigma_{k,\ell}(\Omega_\varepsilon)=m-n-2+j\,.
\end{gather}

In particular, for $n=m-2$ and $\ell=0$, this limit is 0. In that case, the following improvement holds for each $k>0$,
\begin{gather}\label{eq:logimprovement}
  \lim\limits_{\varepsilon\to0}\varepsilon|\log\varepsilon|\sigma_{k,0}(\Omega_\varepsilon)=1\,.
\end{gather}
\end{theoreme}

This theorem shows that in the limit $\eps\to 0$, the Steklov spectrum collapses to a sequence of infinite diverging clusters indexed by the parameter $j\geq 0$. It is remarkable that the geometry of the submanifold $N\subset M$ has no influence on this limit behaviour. Indeed the only remaining information related to $N$ is its codimension $m-n$.

\begin{remarque}
If $N$ is not connected, equation \eqref{eq:mainresultIntro} stays true for each eigenvalue whose index is larger than the number of boundary component of $\Omega_\eps$. Indeed, the proof works for those eigenvalues. The general case will be considered elsewhere.
\end{remarque}

In the limit $\eps\to 0$ the ordered eigenvalues $\sigma_k(\Omega_\eps)$ correspond to the smallest cluster, at $j=0$.
\begin{corollaire}\label{cor:introsigmak}
  Let $M$ be a smooth compact Riemannian manifold of dimension $m\geq 3$ and let $N\subset M$ be a smooth closed connected embedded submanifold of dimension $0<n\leq m-2$.
  Then for each $k\in\N$,
  \begin{gather}\label{eq:limitsigmakintro}
    \lim\limits_{\eps\to0}\eps\sigma_{k}(\Omega_\varepsilon)=m-n-2.
  \end{gather}
  Moreover, in the case where $n=m-2$, the following holds for each $k\in\N$,
  \[\lim\limits_{\eps\to0}\eps|\log\varepsilon|\sigma_{k}(\Omega_\varepsilon)=1\,.\]
\end{corollaire}
The case where the submanifold is a point is excluded from Theorem \ref{thm:PrincipalIntro}. The limit behaviour in this case is given in the next result.

\begin{theoreme}\label{thm:pointIntro}
Let $M$ be a smooth compact Riemannian manifold of dimension $m\geq 2$ and $p\in M$. Then, if $j=0$, set $k=0$ and for $j>0$, choose the unique $k>0$ such $m_0+\cdots+m_{k-1}\leq j< m_0+\cdots+m_{k-1}+m_k$. The following limit holds
\[\lim\limits_{\varepsilon\to0}\eps\sigma_j(\Omega_\eps)=m+k-2\,,\]
where $\Omega_\eps:=M\setminus B(p,\eps)$.

\end{theoreme}

\begin{remarque}
For a compact submanifold of dimension $m-1$, the behavior of the Steklov eigenvalues is completely different.
An example is presented in Section \ref{section:main}.
\end{remarque}

\subsection{Application to isoperimetric type problem}
Given a complete Riemannian manifold $M$ of dimension $m\geq 2$, the question to find upper bounds for  $\sigma_1(\Omega)|\partial\Omega|^{1/(m-1)}$ among bounded domains $\Omega\subset M$ has a rich history. In the Euclidean space $M=\R^m$ this question is equivalent to the maximization of $\sigma_1$ among domains with prescribed boundary measure $|\partial\Omega|$. For $m=2$ the optimal upper bound is known thanks to~\cite{wein, kok,GirKarpLag}, while for $m\geq 3$ it is known that $\sigma_1(\Omega)|\partial\Omega|^{1/(m-1)}$ is bounded above~\cite{ceg2}, but the optimal bound remains unknown. For domains $\Omega$ in a compact manifold of dimension $m\geq 3$, the situation is completely different: it was proved in~\cite{girouard2020large} that in that case $\sigma_1(\Omega)|\partial\Omega|^{1/(m-1)}$ is not bounded above. The proof relies on an homogenization procedure, in which a domain $\Omega_\eps\subset M$ is obtained by removing an unbounded number of uniformly distributed small balls from the compact manifold. Theorem~\ref{thm:PrincipalIntro} leads to an alternative and simpler approach.
\begin{corollaire}\label{cor:valeurproprenormalisee}
  Let $M$ be a compact Riemannian manifold of dimension $m\geq 3$ and let $N\subset M$ be a closed connected submanifold of dimension $0<n\leq m-2$. Then the domains $\Omega_\eps\subset M$ defined by~\eqref{def:perforateDom} satisfy
\[\lim\limits_{\varepsilon\to0}\sigma_1(\Omega_\eps)|\partial\Omega_\varepsilon|^{1/(m-1)}=+\infty\,.\]
\end{corollaire}
\begin{proof}
 Because the volume of the boundary behaves as $|\partial\Omega_\eps|\sim c\eps^{m-n-1}$, if $n\neq m-2$ it follows from~\eqref{eq:mainresultIntro} that
  \begin{align*}
    \sigma_1(\Omega_\eps)|\partial\Omega_\eps|^{1/(m-1)}\sim c\sigma_1(\Omega_\eps)\eps^{\frac{m-n-1}{m-1}}
    =c\eps\sigma_1(\Omega_\eps)\eps^{\frac{-n}{m-1}}\sim  (m-n-2)\eps^{\frac{-n}{m-1}}\to\infty.
  \end{align*}
  If $n=m-2$, it follows from~\eqref{eq:logimprovement} that
  \begin{gather*}
     \sigma_1(\Omega_\varepsilon)|\partial\Omega_\varepsilon|^{1/(m-1)}\sim c\sigma_1(\Omega_\varepsilon)\varepsilon^{\frac{m-n-1}{m-1}} =c\varepsilon|\log\varepsilon|\sigma_1(\Omega_\varepsilon)\frac{\varepsilon^{-\frac{m-2}{m-1}}}{|\log\varepsilon|}\sim c\frac{\varepsilon^{-\frac{m-2}{m-1}}}{|\log\varepsilon|}\to\infty\,.
\end{gather*}
\end{proof}

\begin{remarque}
  The behaviour of $\sigma_1(\Omega_\eps)$ for a point does not lead to divergence. Indeed, if $N$ is a point, then it follows from Theorem~\ref{thm:pointIntro} that
  \begin{gather*}
    \sigma_1(\Omega_\eps)|\partial\Omega_\eps|^{1/(m-1)}\sim c\sigma_1(\Omega_\eps)\eps \xrightarrow{\eps\to 0}c(m-1).
  \end{gather*}  
\end{remarque}

\begin{remarque}
Another normalization that interests some authors is $\sigma_{k}(\Omega)\frac{|\partial\Omega|}{|\Omega|^{\frac{m-2}{m}}}$, see~\cite{GirKarpLag}. In this case, the normalisation does not lead to divergence. Indeed, since $|\partial\Omega_\eps|\sim c\eps^{m-n-1}$ and $|\Omega_\eps|\xrightarrow{\eps\to 0}|\Omega|$, it follows from \eqref{eq:limitsigmakintro} that
\[\sigma_{k}(\Omega_\eps)\frac{|\partial\Omega_\eps|}{|\Omega_\eps|^{\frac{m-2}{m}}}\sim\frac{c\eps\sigma_{k}(\Omega_\eps)\eps^{m-n-2}}{|\Omega|^{\frac{m-2}{m}}}\sim\frac{c(m-n-2)\eps^{m-n-2}}{|\Omega|^{\frac{m-2}{m}}}\xrightarrow{\eps\to 0}0\,.\]
\end{remarque}

\subsubsection{Upper bound involving the intersection index and injectivity radius of the boundary} Corollary~\ref{cor:valeurproprenormalisee} provides a new family of manifolds with large Steklov spectral gap.  By comparing it with known upper bounds one can investigate the necessity of various geometric quantities it involves. For instance, in the recent paper~\cite{colbois2020upper}, Colbois and Gittins have provided upper bounds for the Steklov eigenvalues $\sigma_k$ of submanifolds $\Omega^m$ in $\R^d$ in terms of an intersection index $i(\Omega)$ which counts the number of intersection between $\Omega$ and a generic $p$-plane $\Pi\subset\R^d$, where $p=d-m$. Their bounds also involve the injectivity radius $\mbox{inj}(\partial\Omega)$ of the boundary, as well as its volume:
\begin{gather}\label{ineq:ColboisGittins}
  \sigma_k(\Omega)\leq{A}(m)\frac{i(\Omega)}{\mbox{inj}(\partial\Omega)}+{B}(m)i(\Omega)\bigg(\frac{i(\partial\Omega)k}{|\partial\Omega|}\bigg)^{1/m-1}.
\end{gather}
We show that the presence of the injectivity radius in the denominator of the first term in the right-hand-side of~\eqref{ineq:ColboisGittins} is essential.
Let $M\subset\R^{m+1}$ be any closed hypersurface, with $m\geq 4$. Let $N\subset M$ be a closed curve and consider our usual $\Omega_\eps\subset M$ as defined in~\eqref{def:perforateDom}. Apply inequality~\eqref{ineq:ColboisGittins} to $\Omega_\eps$ and multiply by $\eps>0$ on both sides to obtain:
\begin{gather}
  \eps\sigma_1(\Omega_\eps)\leq{A}(m)\frac{\eps i(\Omega_\eps)}{\mbox{inj}(\partial\Omega_\eps)}+{B}(m)i(\Omega_\eps)\eps\bigg(\frac{i(\partial\Omega_\eps)}{|\partial\Omega_\eps|}\bigg)^{1/m-1}.
\end{gather}
  It follows from Corollary~\ref{cor:valeurproprenormalisee} that $\eps\sigma_1(\Omega_\eps)\xrightarrow{\eps\to0} m-3$, while the intersection indices $i(\Omega_\eps)$ and $i(\partial\Omega_\eps)$ are uniformly bounded and the volume of the boundary satisfies $|\partial\Omega_\eps|\sim c\eps^{m-2}$. Hence there is a constant $K$ such that
\begin{gather*}
  m-3\leq K\limsup_{\eps\to 0}\frac{\eps}{\mbox{inj}(\partial\Omega_\eps)}+\eps^{1/(m-1)}.
\end{gather*}
If the injectivity radius did not occur in~\eqref{ineq:ColboisGittins}, then it would also not occur in this last inequality and the right-hand-side would tend to 0, which is impossible because $m\geq 4$.
By rescaling this construction, we obtain the following result.
\begin{corollaire}
  For $m\geq 3$, there exits a family of smooth hypersurfaces $\Omega_\eps\subset\R^{m+1}$ such that $|\partial\Omega_\eps|=1$ with $i(\Omega_\eps)$ and $i(\partial\Omega_\eps)$ bounded and with $\sigma_1(\Omega_\eps)\xrightarrow{\eps\to0}\infty$.
\end{corollaire}
\begin{remarque}
  In their paper~\cite{colbois2020upper}, Colbois and Gittins also presented an example which proves the necessity of a first term which involves the injectivity radius. Their example is more specific to this task and we feel that our construction is more flexible. See~\cite{ColboisGirouard2021} for another recent application of Theorem~\ref{thm:PrincipalIntro}.

\end{remarque}
\subsection{Discussion and existing literature}
The behaviour of Steklov eigenvalues under small excision has already been studied in various contexts. In their paper~\cite{FraserSchoen}, Fraser and Schoen considered a perforation of a manifold with boundary using a tubular neighborhood of a curve that connects distinct points on the boundary. See also~\cite{Hong} for similar higher-dimensional surgeries. A particularly important inspiration for the current project was the recent paper~\cite{CHIADOPIAT2020100007} by Chiad\`o Piat and Nazarov, in which they consider the excision of a compact domain $\Omega\subset\R^3$ containing the origin by thin tubular neighborhoods of a closed planar curve that is contained in the planar section $\Omega\cap\{z=0\}$. In their work, the cross-section of the tube does not have to be circular. Rather, it is described by a bounded open set $\omega\subset\R^2$. For mixed Steklov-Neumann eigenvalues in the mid-frequency range $\{\sigma\in\lbrack0,+\infty)~|~\sigma<c\varepsilon^{-1}\}$ they prove
\[\lim\limits_{\varepsilon\to0}\varepsilon|\log\varepsilon|\sigma_k^\varepsilon=\frac{2\pi}{|\partial\omega|}\,.\]
For the unit disk $\omega=\D$ this coincides with our asymptotic~\eqref{eq:logimprovement}.
While the method of~\cite{CHIADOPIAT2020100007} leads to more precision (full asymptotic expansions are proved), our Theorem~\ref{thm:PrincipalIntro} applies to a much more general geometric context. Moreover, the proof of Theorem~\ref{thm:PrincipalIntro} is very simple in comparison to the pseudodifferential techniques that are developped in~\cite{CHIADOPIAT2020100007}, and they lead to convergence results for the full spectrum rather than for eigenvalues in the mid-frequency range.

\subsection{Plan of the paper}
In Section~\ref{section:quasiiso}, we use Fermi coordinates to show that any closed submanifold $N\subset M$ admits tubular neighborhoods that are quasi-isometric to products. This allows the comparison of Steklov eigenvalues with mixed Steklov-Dirichlet and Steklov-Neumann eigenvalues on these products. These are then computed, in Section \ref{section:brack}, using separation of variables. The resulting mixed eigenvalues are expressed in terms of $\eps$ and of the codimension of $N$ in $M$ and this allows the proof of the main result in Section \ref{section:main}.

\section{Quasi-isometry}\label{section:quasiiso}
The proof of Theorem~\ref{thm:PrincipalIntro} is based on comparison between Steklov eigenvalues of $\Omega_\eps$ with eigenvalues of mixed Steklov-Dirichlet and Steklov-Neumann problems on tubular neighborhoods $T_\eps$ of the submanifolds $N$. For these problems, separation of variables makes it possible to compute the spectrum explicitly for a Riemannian metric that is comparable to the orginial metric $g$ in the sense of quasi-isometries.
\begin{definition}\label{def:quasiiso}
Let $g_1\,,g_2$ be two Riemannian metrics on a given manifold $M$. We say that $g_1$ and $g_2$ are quasi-isometric with constant $K\geq 1$ if for all $p\in M$ and for all $v\in T_pM\backslash\{0\}$,
\[\frac{1}{K}\leq\frac{g_1(v,v)}{g_2(v,v)}\leq K\,.\]
\end{definition}
The next proposition shows that any submanifold $N^n\subset M^m$ of positive codimension admits a neighbourhood which is quasi-isometric to a cylinder $N\times\mathbb{B}^{m-n}(\delta)$ with a constant that is arbitrarily close to $1$.
\begin{proposition}\label{prop:quasi-iso}
Let $(M,g)$ be an $m-$dimensional Riemannian compact manifold and $N\subset M$ a compact submanifold of dimension $n<m$. For every $\varepsilon_0>0$, there exists $\delta>0$ such that, on $\{p\in M~|~d_g(p,N)<\delta\}$, $g$ is quasi-isometric to the product metric $\Tilde{g}:= h\oplus g_E$ with constant $1+\varepsilon_0$. Here, $h$ is the restriction of $g$ to $N$ and $g_E$ is the $(m-n)-$dimensional Euclidean metric.
\end{proposition}

The proof of Proposition~\ref{prop:quasi-iso} is based on the use of Fermi coordinates along the submanifold $N\subset M$. The Fermi coordinates are a generalization of normal coordinates. Given a point $p\in N$, there exist a system of coordinates $(y_1,\ldots,y_n;U)$ on a open neighbourhood $U\subset N$ containing $p$, a $\delta>0$ and a small neighbourhood $\mathcal{O}\subset \{(q,v)~|~ q\in U\text{ and } v\in T_qM\,,v\perp T_qN\}$ such that the exponential map $\exp|_\mathcal{O}:\mathcal{O}\to T_\delta$ is a diffeomorphism, with $T_\delta:=\{p\in M~|~d_g(p,N)<\delta\}$. 
The Fermi coordinates around $p$ are given by $(y_1,\ldots,y_n,\exp_{(y_1,\ldots,y_n)}^{-1})$. See \cite{tubes} for a concise presentation of these coordinates and their fundamental properties.

\begin{proof}
Let us recall that for $s>0$, $T_s:=\{p\in M~|~d_g(p,N)<s\}$.

Let $\varepsilon_0>0$. 
Let $(x_1,\ldots,x_m)$ be the Fermi coordinates around a point $p\in N$ on an open set $U\subset M$. Then, on $N\cap U$, $x_1,\ldots,x_n$ form a system of coordinates on $N$. Moreover, on $N\cap U$, the vector fields $\frac{\partial}{\partial x_i}$, for $i=n+1,\ldots,m$, are orthonormal. 
Thus, for every $p\in N\cap U$, the metric $g$ is of the form
\begin{gather}\label{eq:metricFermi}
  g_{ij}(p)=
  \begin{cases}
    h_{ij}&\mbox{ for $1\leq i,j\leq n$}\,, \\
    0&\mbox{ for $1\leq i \leq n$ and $n+1\leq j\leq m$}\,,\\
    \delta_{ij}&\mbox{ for $n+1\leq i,j\leq m$}\,.
  \end{cases}
\end{gather}
Let $\tilde{g}:=h\oplus g_E$ defined on $U\subset M$. In other words, the same formula~\eqref{eq:metricFermi} is used for all $x\in U$:
  \[
    \tilde{g}_{ij}(x)=
    \begin{cases}
      h_{ij}&\mbox{ for $1\leq i,j\leq n$}\,, \\
      0&\mbox{ for $1\leq i \leq n$ and $n+1\leq j\leq m$}\,,\\
      \delta_{ij}&\mbox{ for $n+1\leq i,j\leq m$}\,.
    \end{cases}
  \]

There exists $L\in\N$ such that $1-\varepsilon_0/L\geq1/(1+\varepsilon_0)$.
By continuity of $g$, there exists $\delta>0$ such that if $q\in U$ and $q\in T_{3\delta}$, then
\[|g_{ij}(q)-\Tilde{g}_{ij}(q)|<\varepsilon_0/L\,.\]
For $v\in T_qM$ such that $\Tilde{g}(v,v)=1$, 
\[\frac{1}{1+\varepsilon_0}\leq 1-\varepsilon_0/L\leq g(v,v)\leq 1+\varepsilon_0/L\leq 1+\varepsilon_0\,.\]
By linearity of $g$, it follows that
\[\frac{1}{1+\varepsilon_0}\leq\frac{g(v,v)}{\Tilde{g}(v,v)}\leq 1+\varepsilon_0\,,\]
for every $v\in T_qM\backslash\{0\}$.

Let $\chi\in C^\infty(M)$ be such that 
\begin{gather*}
    0\leq\chi(x)\leq 1\text{ for all }x\in M\,,\\
    \chi\equiv1 \text{ in } T_{\delta}\,,\\
     \chi\equiv0 \text{ in } M\backslash T_{3\delta/2}\,.
\end{gather*}
Define
\[\overline{g}:=\begin{cases}
(1-\chi)g+\chi\Tilde{g}&\mbox{ on $T_{3\delta}$},\\
g&\mbox{ elsewhere }.
\end{cases}\]
By the previous computation, $\overline{g}$ and $g$ are quasi-isometric with constant $1+\varepsilon_0$ on $M$. Moreover, on $T_\delta$, since $\overline{g}=\Tilde{g}$, it follows that $\Tilde{g}$ and $g$ are quasi-isometric with constant $1+\varepsilon_0$.
\end{proof}

\begin{remarque}\label{remarque:quasi-isopoint}
In the case where $N$ is a point $p$, the Fermi coordinates around $p$ are the normal coordinates $x_1,\ldots,x_m$ given by the inverse of the exponential map. For each $i$, the vector fields $\frac{\partial}{\partial x_i}\bigg|_p$ satisfy
\[\frac{\partial}{\partial x_i}\bigg|_p=d(\exp_p)_0(e_i)=e_i\,.\]
Thus, centered at $p$, the metric $g$ is the Euclidean metric $g_E$.
By a similar argument as seen previously, we show that $g$ is quasi-isometric to $g_E$ with constant $1+\varepsilon_0$ over $B_\delta(p)$.
\end{remarque}

The following proposition is borrowed from \cite[Proposition 2.2]{CEG3}.
\begin{proposition}\label{prop:quasiisovp}
Let $M$ be a Riemannian manifold of dimension $m$. Let $g_1\,,g_2$ be two Riemannian metrics on $M$ which are quasi-isometric with constant $K$. The Steklov eigenvalues with respect to $g_1$ and to $g_2$ satisfy the following inequality
\[\frac{1}{K^{m+1/2}}\leq\frac{\sigma_k(M,g_1)}{\sigma_k(M,g_2)}\leq K^{m+1/2}\,.\]
\end{proposition}

\section{Mixed Steklov problems}\label{section:brack}

Let $\Omega$ be a Riemannian manifold with boundary $\partial\Omega$. Consider an open neighborhood of the boundary $A\subsetneq\Omega$. In other words $A\subsetneq\Omega$ is open and satisfies $\partial\Omega=A\cap\partial\Omega$. Let $\Sigma:=\partial A\setminus\partial\Omega$ be the inner part of the boundary of $A$.
We will use the following mixed Steklov-Dirichlet and Steklov-Neumann:
\begin{equation*}
    \begin{cases}
    \Delta u=0&\mbox{ in $A$}\,,\\
    \partial_n u=0 &\mbox{ on $\Sigma$}\,,\\
    \partial_n u=\sigma^N u &\mbox{ on $\partial\Omega$}\,,
  \end{cases}
  \qquad\qquad\text{ and }\qquad\qquad
    \begin{cases}
    \Delta u=0&\mbox{ in $A$}\,,\\
    u=0 &\mbox{ on $\Sigma$}\,,\\
    \partial_n u=\sigma^N u &\mbox{ on $\partial\Omega$}\,.
    \end{cases}
\end{equation*}
Their spectra are given by unbounded sequences of eigenvalues
$0=\sigma_0^N<\sigma_1^N(A)\leq\sigma_2^N(A)\leq\cdots$ and
$0<\sigma_1^D\leq\sigma_2^D(A)\leq\sigma_3^D(A)\leq\cdots$ and it follows from their variational characterizations that for all $j\geq0$, the following inequality holds:
\[\sigma_j^N(A)\leq\sigma_j(\Omega)\leq\sigma_{j+1}^D(A)\,.\]
This is a classical application of the Dirichlet--Neumann bracketing. See  \cite[Section 2]{colbois2018steklov} for details.

\subsection{Steklov-Dirichlet problem on products}

As seen previously, a tubular neighbourhood of $N$ is quasi-isometric to the product manifold $N\times\mathbb{B}^{m-n}(\delta)$. We will study the Steklov-Dirichlet problem on the manifold $N\times\lbrack\varepsilon,\delta\rbrack\times\s^{m-n-1}$ equipped with the Riemannian metric 
\[h\oplus dr^2\oplus r^2g_0\,,\]
where $h$ is the metric on $N$ and $g_0$ is the round metric on $\s^{m-n-1}$.

\begin{lemme}\label{lemme:dirichlet}
The spectrum of the mixed problem
\begin{equation}\label{eq:steklov_dirichlet}
    \begin{cases}
    \Delta u=0 & \mbox{ in } N\times(\varepsilon,\delta)\times\s^{m-n-1}\,,\\
    \partial_n u=\sigma^D u &\mbox{ on } N\times\{\varepsilon\}\times\s^{m-n-1}\,,\\
    u=0&\mbox{ on } N\times\{\delta\}\times\s^{m-n-1}\,,
    \end{cases}
\end{equation}
is given by
 \[\bigsqcup_{k\geq0}\sigma(\Lambda^{\lambda_k,D})\,,\]
where $\lambda_k$ is the $k$-th eigenvalue of the Laplacian on $N$
and $\sigma(\Lambda^{\lambda_k,D})$ is the spectrum of the operator $\Lambda^{\lambda_k,D}:C^\infty(\{\varepsilon\}\times\s^{m-n-1})\to C^\infty(\{\varepsilon\}\times\s^{m-n-1})$ defined by
\[\Lambda^{\lambda_k,D}(g):=\partial_n G\,,\]
where $G$ is the unique solution of the problem
 \begin{equation}\label{eq:problemeDN}
 \begin{cases}
\Delta G= \lambda_k G & \mbox{ in $(\varepsilon,\delta)\times\s^{m-n-1}$}\,,\\
G=g & \mbox{ on $\{\varepsilon\}\times\s^{m-n-1}$}\,,\\
G=0 & \mbox{ on $\{\delta\}\times\s^{m-n-1}$}\,.
\end{cases}
\end{equation}
\end{lemme}

\begin{remarque}\label{rem:dirichlet}
The convention used for the Laplacian on a Riemannian manifold $(\Omega,g)$ is given by the formula
\[\Delta u=\frac{1}{\sqrt{\det G}}\partial_i\big(\sqrt{\det G}g^{ij}\partial_j u\big)\,,\]
where $G$ is the matrix form of the Riemannian metric $g$ et $g^{ij}$ are the components of the inverse of $G$.
There exists a unique solution to the problem \eqref{eq:problemeDN} because $-\lambda_k$ is not an eigenvalue of the Dirichlet problem on $(\varepsilon,\delta)\times\s^{m-n-1}$. Indeed, the Dirichlet eigenvalues are positives and $-\lambda_k$ is nonpositive.
\end{remarque}

\begin{proof}
First of all, the Laplace-Beltrami operator on $N\times(\varepsilon,\delta)\times\s^{m-n-1}$ with the given metric $dr^2\oplus h\oplus r^2g_0$ is given by
\[\Delta u=\Delta_N u+\frac{m-n-1}{r}\partial_r u+\partial_{rr}u+\frac{1}{r^2}\Delta_{\s^{m-n-1}}u\,.\]

Suppose that the solution of problem \eqref{eq:steklov_dirichlet} is of the form $u(p,r,q)=F(p)G(r,q)$.
Then, $\Delta u=0$ becomes
\[\frac{-\Delta_N F}{F}=\frac{\frac{m-n-1}{r}\partial_r G+\partial_{rr}G+\frac{1}{r^2}\Delta_{\s^{m-n-1}}G}{G}=\lambda\,,\]
for some $\lambda\in\R$.

The equation $-\Delta_N F=\lambda F$ is the Laplace equation on $N$ which gives the solution $(\lambda_k,F_k)_{k\geq0}$ with the convention that $\lambda_0=0$.
 
For all $k\geq0$, we have to solve the problem
\[\begin{cases}
\frac{m-n-1}{r}\partial_r G+\partial_{rr}G+\frac{1}{r^2}\Delta_{\s^{m-n-1}}G= \lambda_k G & \mbox{ in $(\varepsilon,\delta)\times\s^{m-n-1}$}\,,\\
\partial_n G=\sigma^D G & \mbox{ on $\{\varepsilon\}\times\s^{m-n-1}$}\,,\\
G=0 & \mbox{ on $\{\delta\}\times\s^{m-n-1}$}\,.
\end{cases}\]
This is the spectral problem associated to the operator $\Lambda^{\lambda_k,D}$.
Thus, it becomes clear that the spectrum of \eqref{eq:steklov_dirichlet} is given by
 \[\bigsqcup_{k\geq0}\sigma(\Lambda^{\lambda_k,D})\,.\]
\end{proof}

As seen in the proof of Lemma \ref{lemme:dirichlet}, to find the spectrum, we need to solve problem \eqref{eq:problemeDN}.
When solving this problem using separation of variables as it will be seen in proof of Lemma \ref{lemme:asimptotiquedirichlet}, we encounter the differential equation
\[x^2R''+xR'-(x^2+\nu^2)R=0\,,\]
whose solutions are called the modified Bessel functions $I_\nu(x)\,, K_\nu(x)$. The differential equation is obtained by replacing $x$ by $\pm ix$ in Bessel's equation (see \cite[Chapter 10.25]{NIST:DLMF} for further information).
In the proof of Lemma \ref{lemme:asimptotiquedirichlet}, we use the following recurrence relations (see \cite[Chapter 10.29]{NIST:DLMF})
\begin{equation}\label{eq:deriveI0}
I_0'(x)=I_1(x)\,,
\end{equation}
\begin{equation}\label{eq:deriveK0}
K_0'(x)=-K_1(x)\,,
\end{equation}
\begin{equation}\label{eq:deriveI}
I_\nu'(x)=I_{\nu-1}(x)-\frac{\nu}{x}I_\nu(x)\,,
\end{equation}
\begin{equation}\label{eq:deriveK}
    K_\nu'(x)=\frac{\nu}{x}K_\nu(x)-K_{\nu+1}(x)\,.
\end{equation}
We also use the following asymptotics, which hold as $x\to0$, (see \cite[Chapter 10.30]{NIST:DLMF})
\begin{equation}\label{eq:asymptotiqueK0}
K_0(x)\sim-\log{x}\,,
\end{equation}
\begin{equation}\label{eq:asymptotiqueI}
I_\nu(x)\sim\frac{(\frac{1}{2}x)^\nu}{\Gamma(\nu+1)}\,,
\end{equation}
\begin{equation}\label{eq:asymptotiqueK}
    K_\nu(x)\sim\frac{\frac{1}{2}\Gamma(\nu)}{(\frac{1}{2}x)^\nu}\,.
\end{equation}

\begin{lemme}\label{lemme:asimptotiquedirichlet}
We have the following asymptotics for the distinct eigenvalues of problem \eqref{eq:problemeDN}:
\newline If $n=m-2$,
\begin{gather*}
    \sigma_{k,0}^D\sim\frac{1}{\varepsilon|\log\varepsilon|}\mbox{ if $k\geq0$}\,,\\
    \sigma_{k,(j)}^D\sim\frac{j}{\varepsilon}\mbox{ if $k=0$ and $\ell>0$ or if $k\,,\ell\neq0$}\,.
\end{gather*}
If $n\neq m-2$,
\begin{gather*}
    \sigma_{k,(j)}^D\sim\frac{m-n-2+j}{\varepsilon}\mbox{ for every $k\,,j\geq0$}\,.
\end{gather*}
\end{lemme}

\begin{proof}
Suppose that $G(r,q)=R(r)\phi(q)$. 
Then, we have
\[\frac{r^2R''+(m-n-1)rR'-r^2\lambda_k R}{R}=\frac{-\Delta_{\s^{m-n-1}}\phi}{\phi}=\mu\,,\]
for some $\mu\in\R$.

The equation $-\Delta_{\s^{m-n-1}}\phi=\mu\phi$ admits a solution if $\mu=\mu_{(j)}=j(j+m-n-2)$. The notation $\mu_{(j)}$ means that the distinct eigenvalues are considered. The multiplicity of the eigenvalue $\mu_{(j)}$ is $m_j$.

Finally, for $k\,,j\geq0$, we solve the problem
\begin{equation*}
    \begin{cases}
    r^2R''+(m-n-1)rR'-(r^2\lambda_k+\mu_{(j)})R=0\,,\\
    -R'(\varepsilon)=\sigma^D R(\varepsilon)\,,\\
    R(\delta)=0\,.
    \end{cases}
\end{equation*}

Let us consider two cases:
\begin{enumerate}
    \item $n=m-2$\,,
    \item $n\neq m-2$\,.
\end{enumerate}

\textbf{Case 1: $n=m-2$}
\newline If $k=j=0$, 
\[r^2R''+rR'=0\,,\]
whose solution is $R(r)=a+b\log r$. 
With the condition $R(\delta)=0$, we obtain $a=-b\log\delta$.
With the condition $-R'(\varepsilon)=\sigma^D(\varepsilon)$, we obtain
\[\sigma_{0,0}^D=\frac{1}{\varepsilon\log(\delta/\varepsilon)}\sim\frac{1}{\varepsilon|\log\varepsilon|}\,.\]

If $k=0$ and $j\neq0$, we have
\[r^2R''+rR'-\mu_{(j)} R=0\,,\]
whose solution is $R(r)=ar^j+br^{-j}$.
With the condition $R(\delta)=0$, we obtain $a=-b\delta^{-2j}$.
With the condition $-R'(\varepsilon)=\sigma^D R(\varepsilon)$, we obtain the distinct eigenvalue
\[\sigma_{0,{(j)}}^D=\frac{j(1+\delta^{-2j}\varepsilon^{2j})}{\varepsilon(1-\delta^{-2j}\varepsilon^{2j})}\sim\frac{j}{\varepsilon}\,.\]
Since the multiplicity of $\mu_{(j)}$ is $m_j$, the multiplicity of $\sigma^D_{0,(j)}$ is also $m_j$.

If $k\neq0$ and $j$ is arbitrary, we put $x:=\sqrt{\lambda_k}r$ to obtain
\[x^2R''+xR'-(x^2+j^2)R=0\,,\]
whose solutions are the modified Bessel functions $I_j$ et $K_j$.
So, $R(r)=aI_j(\sqrt{\lambda_k}r)+bK_j(\sqrt{\lambda_k}r)$.
With the condition $R(\delta)=0$, we obtain $a=-b\frac{K_j(\sqrt{\lambda_k}\delta)}{I_j(\sqrt{\lambda_k}\delta)}$.
With the condition $-R'(\varepsilon)=\sigma^D R(\varepsilon)$, we obtain
\[\sigma_{k,(j)}^D=\frac{\sqrt{\lambda_k}\big(\frac{K_j(\sqrt{\lambda_k}\delta)}{I_j(\sqrt{\lambda_k}\delta)}I_j'(\sqrt{\lambda_k}\varepsilon)-K_j'(\sqrt{\lambda_k}\varepsilon)\big)}{K_j(\sqrt{\lambda_k}\varepsilon)-\frac{K_j(\sqrt{\lambda_k}\delta)}{I_j(\sqrt{\lambda_k}\delta)}I_j(\sqrt{\lambda_k}\varepsilon)}\,.\]
Since the multiplicity of $\mu_{(j)}$ is $m_j$, the multiplicity of $\sigma^D_{k,(j)}$ is also $m_j$.

Using the asymptotics \eqref{eq:asymptotiqueK0}, \eqref{eq:asymptotiqueI} with $\nu=0$, we have
\[\frac{K_0(\sqrt{\lambda_k}\delta)}{I_0(\sqrt{\lambda_k}\delta)}\sim -\log(\sqrt{\lambda_k}\delta)\,.\]

Using the asymptotics \eqref{eq:asymptotiqueI}, \eqref{eq:asymptotiqueK}, we have
\[\frac{K_j(\sqrt{\lambda_k}\delta)}{I_j(\sqrt{\lambda_k}\delta)}\sim \frac{1}{2}\Gamma(j)\Gamma(j+1)\bigg(\frac{1}{2}\sqrt{\lambda_k}\delta\bigg)^{-2j}\,.\]

Thus, with the recurrence relations \eqref{eq:deriveI0}, \eqref{eq:deriveK0},
\begin{gather*}
\begin{aligned}
    \sigma_{k,0}^D&=\frac{\sqrt{\lambda_k}\big(\frac{K_0(\sqrt{\lambda_k}\delta)}{I_0(\sqrt{\lambda_k}\delta)}I_0'(\sqrt{\lambda_k}\varepsilon)-K_0'(\sqrt{\lambda_k}\varepsilon)\big)}{K_0(\sqrt{\lambda_k}\varepsilon)-\frac{K_0(\sqrt{\lambda_k}\delta)}{I_0(\sqrt{\lambda_k}\delta)}I_0(\sqrt{\lambda_k}\varepsilon)}\\
    &\sim \frac{\sqrt{\lambda_k}(-\log(\sqrt{\lambda_k}\delta)I_1(\sqrt{\lambda_k}\varepsilon)+K_1(\sqrt{\lambda_k}\varepsilon))}{K_0(\sqrt{\lambda_k}\varepsilon)+\log(\sqrt{\lambda_k}\delta)I_0(\sqrt{\lambda_k}\varepsilon)}\\
    &\sim \frac{1-\frac{1}{2}\log(\sqrt{\lambda_k}\delta)(\sqrt{\lambda_k}\varepsilon)^2}{\varepsilon\log(\delta/\varepsilon)}\sim\frac{1}{\varepsilon|\log\varepsilon|}\,.
\end{aligned}
\end{gather*}

With the recurrence relations \eqref{eq:deriveI}, \eqref{eq:deriveK}, 
\begin{gather*}
\begin{aligned}
    \sigma_{k,(j)}^D&=\frac{\sqrt{\lambda_k}\big(\frac{K_j(\sqrt{\lambda_k}\delta)}{I_j(\sqrt{\lambda_k}\delta)}I_j'(\sqrt{\lambda_k}\varepsilon)-K_j'(\sqrt{\lambda_k}\varepsilon)\big)}{K_j(\sqrt{\lambda_k}\varepsilon)-\frac{K_j(\sqrt{\lambda_k}\delta)}{I_j(\sqrt{\lambda_k}\delta)}I_j(\sqrt{\lambda_k}\varepsilon)}\\
    &\sim \frac{\sqrt{\lambda_k}\Gamma(j+1)4^{-1}((\frac{1}{2}\sqrt{\lambda_k}\varepsilon)^{-j-1}+(\frac{1}{2}\sqrt{\lambda_k}\delta)^{-2j}(\frac{1}{2}\sqrt{\lambda_k}\varepsilon)^{j-1})}{2^{-1}\Gamma(j)(\frac{1}{2}\sqrt{\lambda_k}\varepsilon)^{-j}(1-(\frac{1}{2}\sqrt{\lambda_k}\delta)^{-2j}(\frac{1}{2}\sqrt{\lambda_k}\varepsilon)^{2j})}\\
    &=\frac{j(1+(\frac{1}{2}\sqrt{\lambda_k}\delta)^{-2j}(\frac{1}{2}\sqrt{\lambda_k}\varepsilon)^{2j})}{\varepsilon(1-(\frac{1}{2}\sqrt{\lambda_k}\delta)^{-2j}(\frac{1}{2}\sqrt{\lambda_k}\varepsilon)^{2j})}\sim\frac{j}{\varepsilon}\,.
\end{aligned}
\end{gather*}

\textbf{Case 2: $n\neq m-2$}
\newline If $k=j=0$, 
\[r^2R''+(m-n-1)rR'\,,\]
whose solution is $R(r)=a+br^{2+n-m}$.
With the condition $R(\delta)=0$, we obtain $a=-b\delta^{2+n-m}$.
With the condition $-R'(\varepsilon)=\sigma^D R(\varepsilon)$, we obtain
\[\sigma^D_{0,0}=\frac{m-n-2}{\varepsilon(1-\delta^{2+n-m}\varepsilon^{m-n-2})}\sim\frac{m-n-2}{\varepsilon}\,.\]

If $k=0$ and $j\neq0$,we have
\[r^2R''+(m-n-1)rR'-\mu_{(j)}R=0\,,\]
whose solution is $R(r)=ar^j+br^{n+2-m-j}$.
With the condition $R(\delta)=0$, we obtain $a=-b\delta^{n+2-m-2j}$.
With the condition $-R'(\varepsilon)=\sigma^DR(\varepsilon)$, we obtain
\[\sigma_{0,(j)}^D=\frac{m-n-2+j+j\delta^{n+2-m-2j}\varepsilon^{m-n-2+2j}}{\varepsilon(1-\delta^{n+2-m-2j}\varepsilon^{m-n-2+2j})}\sim\frac{m-n-2+j}{\varepsilon}\,.\]
Since the multiplicity of $\mu_{(j)}$ is $m_j$, the multiplicity of $\sigma^D_{0,(j)}$ is also $m_j$.

If $k\neq0$ and $j$ is arbitrary, suppose $R(r)=r^s\eta(r)$ with $\ell=\frac{2+n-m}{2}$. Then, we transform the differential equation in $R$ to the differential equation
\[r^2\eta''+r\eta'-(r^2\lambda_k+\nu^2)\eta=0\,,\]
with $\nu=\frac{m-n-2+2j}{2}$. We know that $\eta(r)=aI_\nu(\sqrt{\lambda_k}r)+bK_\nu(\sqrt{\lambda_k}r)$.
So, $R(r)=r^\ell(aI_\nu(\sqrt{\lambda_k}r)+bK_\nu(\sqrt{\lambda_k}r))$.
With the condition $R(\delta)=0$, we obtain $a=-b\frac{K_\nu(\sqrt{\lambda_k}\delta)}{I_\nu(\sqrt{\lambda_k}\delta)}$.
Thus,
\[R'(r)=b\ell r^{\ell-1}\bigg(K_\nu(\sqrt{\lambda_k}r))-\frac{K_\nu(\sqrt{\lambda_k}\delta)}{I_\nu(\sqrt{\lambda_k}\delta)}I_\nu(\sqrt{\lambda_k}r)\bigg)+b\sqrt{\lambda_k}r^\ell\bigg(K'_\nu(\sqrt{\lambda_k}r))-\frac{K_\nu(\sqrt{\lambda_k}\delta)}{I_\nu(\sqrt{\lambda_k}\delta)}I'_\nu(\sqrt{\lambda_k}r)\bigg)\,.\]
With the recurrence relations \eqref{eq:deriveI}, \eqref{eq:deriveK} and the asymptotics \eqref{eq:asymptotiqueI}, \eqref{eq:asymptotiqueK}, we have
\begin{gather*}
    \frac{K_\nu(\sqrt{\lambda_k}\delta)}{I_\nu(\sqrt{\lambda_k}\delta)}\sim\frac{1}{2}\Gamma(\nu)\Gamma(\nu+1)\bigg(\frac{1}{2}\sqrt{\lambda_k}\delta\bigg)^{-2\nu}\,,\\
    R(\varepsilon)\sim\frac{\Gamma(\nu)\varepsilon^\ell(\frac{1}{2}\sqrt{\lambda_k}\varepsilon)^{-\nu}}{2}\bigg(1-\bigg(\frac{1}{2}\sqrt{\lambda_k}\delta\bigg)^{-2\nu}\bigg(\frac{1}{2}\sqrt{\lambda_k}\varepsilon\bigg)^{2\nu}\bigg)\,,
\end{gather*}
\begin{multline*}
     R'(\varepsilon)\sim\frac{\Gamma(\nu)\varepsilon^{\ell-1}(\frac{1}{2}\sqrt{\lambda_k}\varepsilon)^{-\nu}}{2}\bigg[\nu\bigg(1+\bigg(\frac{1}{2}\sqrt{\lambda_k}\delta\bigg)^{-2\nu}\bigg(\frac{1}{2}\sqrt{\lambda_k}\varepsilon\bigg)^{2\nu}\bigg)
     \\-\ell\bigg(1-\bigg(\frac{1}{2}\sqrt{\lambda_k}\delta\bigg)^{-2\nu}\bigg(\frac{1}{2}\sqrt{\lambda_k}\varepsilon\bigg)^{2\nu}\bigg)\bigg]\,.
\end{multline*}
Thus,
\[\sigma_{k,(j)}^D\sim\frac{\nu-\ell}{\varepsilon}=\frac{m-n-2+j}{\varepsilon}\,.\]
Since the multiplicity of $\mu_{(j)}$ is $m_j$, the multiplicity of $\sigma^D_{k,(j)}$ is also $m_j$.
\end{proof}

\subsection{Steklov-Neumann problem on products}

We will study the Steklov-Neumann problem on the manifold $N\times\lbrack\varepsilon,\delta\rbrack\times\s^{m-n-1}$ equipped with the Riemannian metric 
\[h\oplus dr^2\oplus r^2g_0\,.\]

\begin{lemme}\label{lemme:neumann}
The spectrum of the problem
\begin{equation}\label{eq:steklov_neumann}
    \begin{cases}
    \Delta u=0 & \mbox{ in } N\times (\varepsilon,\delta)\times \s^{m-n-1}\,,\\
    \partial_n u=\sigma^N u &\mbox{ on } N\times\{\varepsilon\}\times \s^{m-n-1}\,,\\
    \partial_n u=0&\mbox{ on }N\times\{\delta\}\times \s^{m-n-1}\,,
    \end{cases}
\end{equation}
is given by
\[\bigsqcup_{k\geq0}\sigma(\Lambda^{\lambda_k,N})\,,\]
where $\lambda_k$ is the $k$-th eigenvalue of the Laplacian on $N$ and $\sigma(\Lambda^{\lambda_k,N})$ is the spectrum of the operator $\Lambda^{\lambda_k,N}:C^\infty(\{\varepsilon\}\times \s^{m-n-1})\to C^\infty(\{\varepsilon\}\times \s^{m-n-1})$ defined by
\[\Lambda^{\lambda_k,N}(g):=\partial_n G\,,\]
where $G$ is the unique solution to the problem
\begin{equation}\label{eq:problemeDN2}
\begin{cases}
\Delta G= \lambda_k G & \mbox{ in $(\varepsilon,\delta)\times \s^{m-n-1}$}\,,\\
G=g & \mbox{ sur $\{\varepsilon\}\times \s^{m-n-1}$}\,,\\
\partial_n G=0 & \mbox{ sur $\{\delta\}\times \s^{m-n-1}$}\,.
\end{cases}
\end{equation}
\end{lemme}

\begin{remarque}\label{rem:neumann}
There exists a unique solution to problem \eqref{eq:problemeDN2} for the same reason given in Remark \ref{rem:dirichlet}.
\end{remarque}

\begin{proof}
Using the same method as in Lemma \ref{lemme:dirichlet}, we need to solve
\[\begin{cases}
\frac{m-n-1}{r}\partial_r G+\partial_{rr}G+\frac{1}{r^2}\Delta_{\s^{m-n-1}}G= \lambda_k G & \mbox{ in $(\varepsilon,\delta)\times\s^{m-n-1}$}\,,\\
\partial_n G=\sigma^D G & \mbox{ on $\{\varepsilon\}\times\s^{m-n-1}$}\,,\\
\partial_n G=0 & \mbox{ on $\{\delta\}\times\s^{m-n-1}$}\,,
\end{cases}\]
for every $k\geq0$.
This is the spectral problem associated to the operator $\Lambda^{\lambda_k,N}$. 
It then becomes clear that the spectrum of problem \eqref{eq:steklov_neumann} is given by
\[\bigsqcup_{k\geq0}\sigma(\Lambda^{\lambda_k,N})\,.\]
\end{proof}

\begin{lemme}\label{lemme:asymptotiqueneumann}
We have the following asymptotics for the distinct eigenvalues of problem \eqref{eq:problemeDN2}:
\newline If $n=m-2$,
\begin{gather*}
    \sigma_{0,0}^N=0\,,\\
    \sigma_{k,(j)}^N\sim\frac{j}{\varepsilon}\mbox{ if $k=0$ and $j\neq0$ or if $k\,,j\neq 0$}\,,\\
    \sigma_{k,0}^N\sim\frac{1}{\varepsilon\bigg(|\log(\sqrt{\lambda_k\varepsilon})|-\frac{K_0'(\sqrt{\lambda_k}\delta)}{I_0'(\sqrt{\lambda_k}\delta)}\bigg)}\mbox{ if $k\neq0$}\,.
\end{gather*}
If $n\neq m-2$,
\begin{gather*}
    \sigma_{0,0}^N=0\,,\\
    \sigma_{k,(j)}^N\sim\frac{m-n-2+j}{\varepsilon} \mbox{if $k=0$ and $j\neq0$ or if $k\neq0$ and $j\geq0$}\,.
\end{gather*}
\end{lemme}

\begin{proof}
Just like the beginning of the proof of Lemma \ref{lemme:asimptotiquedirichlet}, we need to solve the problem
\begin{equation*}
    \begin{cases}
    r^2R''+(m-n-1)rR'-(r^2\lambda_k+\mu_{(j)})R=0\,,\\
    R'(\delta)=0\,,\\
    -R'(\varepsilon)=\sigma^N R(\varepsilon)\,.
    \end{cases}
\end{equation*}

Let us consider two cases:
\begin{enumerate}
    \item $n=m-2$\,,
    \item $n\neq m-2$.
\end{enumerate}

\textbf{Case 1: $n=m-2$}
\newline If $k=j=0$, we have
\[r^2R''+rR'=0\,,\]
whose solution is $R(r)=a+b\log r$.
With the condition condition $R'(\delta)=0$, we obtain $b=0$. 
With the condition $-R'(\varepsilon)=\sigma^N R(\varepsilon)$, we obtain
\[\sigma_{0,0}^N=0\,.\]

If $k=0$ and $j\neq0$, we have
\[r^2R''+rR'-\mu_{(j)}R=0\,,\]
whose solution is $R(r)=ar^j+br^{-j}$.
With the condition $R'(\delta)=0$, we obtain $a=b\delta^{-2m}$. 
with the condition $-R'(\varepsilon)=\sigma^N R(\varepsilon)$, we obtain
\[\sigma_{0,(j)}^N=\frac{j(1-\delta^{-2j}\varepsilon^{2j})}{\varepsilon(1+\delta^{-2j}\varepsilon^{2j})}\sim\frac{j}{\varepsilon}\,.\]
Since the multiplicity of $\mu_{(j)}$ is $m_j$, the multiplicity of $\sigma^N_{0,(j)}$ is also $m_j$.

If $k\neq0$ and $j$ is arbitrary,we have
\[r^2R''+rR'-(r^2\sqrt{\lambda_k}+\mu_{(j)})R=0\,,\]
whose solution is $R(r)=aI_j(\sqrt{\lambda_k}r)+bK_j(\sqrt{\lambda_k}r)$.
With the condition $R'(\delta)=0$, we obtain $a=-b\frac{K_j'(\sqrt{\lambda_k}\delta)}{I_j'(\sqrt{\lambda_k}\delta)}$. 
With the condtion $-R'(\varepsilon)=\sigma^N R(\varepsilon)$, we obtain
\[\sigma_{k,(j)}^N=\frac{\sqrt{\lambda_k}\bigg(\frac{K_j'(\sqrt{\lambda_k}\delta)}{I_j'(\sqrt{\lambda_k}\delta)}I_j'(\sqrt{\lambda_k}\varepsilon)-K_j'(\sqrt{\lambda_k}\varepsilon)\bigg)}{K_j(\sqrt{\lambda_k}\varepsilon)-\frac{K_j'(\sqrt{\lambda_k}\delta)}{I_j'(\sqrt{\lambda_k}\delta)}I_j(\sqrt{\lambda_k}\varepsilon)}\,.\]
Since the multiplicity of $\mu_{(j)}$ is $m_j$, the multiplicity of $\sigma^N_{k,(j)}$ is also $m_j$.
For $j=0$, using  the recurrence relations \eqref{eq:deriveI0}, \eqref{eq:asymptotiqueK0} followed by the asymptotics \eqref{eq:asymptotiqueK0}, \eqref{eq:asymptotiqueI}, \eqref{eq:asymptotiqueI}, \eqref{eq:asymptotiqueK}, we have
\[\sigma_{k,0}^N\sim\frac{\frac{1}{2}\sqrt{\lambda_k}(1+\frac{K'_0(\sqrt{\lambda_k}\delta)}{I'_0(\sqrt{\lambda_k}\delta)}(\frac{1}{2}\sqrt{\lambda_k}\varepsilon)^2)}{\frac{1}{2}\sqrt{\lambda_k}\eps\bigg(|\log(\sqrt{\lambda_k}\eps)|-\frac{K'_0(\sqrt{\lambda_k}\delta)}{I'_0(\sqrt{\lambda_k}\delta)}\bigg)}\sim\frac{1}{\eps\bigg(|\log(\sqrt{\lambda_k}\eps)|-\frac{K'_0(\sqrt{\lambda_k}\delta)}{I'_0(\sqrt{\lambda_k}\delta)}\bigg)}\,.\]

Using the recurrence relations \eqref{eq:deriveI}, \eqref{eq:asymptotiqueK} followed by the asymptotics \eqref{eq:asymptotiqueK}, \eqref{eq:asymptotiqueI}, we have
\[\frac{K_j'(\sqrt{\lambda_k}\delta)}{I_j'(\sqrt{\lambda_k}\delta)}\sim -\frac{1}{2}\Gamma(j)\Gamma(j+1)\bigg(\frac{1}{2}\sqrt{\lambda_k}\delta\bigg)^{-2j}\,.\]

Thus,
\[\sigma_{k,(j)}^N\sim\frac{j(1-(\frac{1}{2}\sqrt{\lambda_k}\delta)^{-2j}(\frac{1}{2}\sqrt{\lambda_k}\varepsilon)^{2j})}{\varepsilon(1+(\frac{1}{2}\sqrt{\lambda_k}\delta)^{-2j}(\frac{1}{2}\sqrt{\lambda_k}\varepsilon)^{2j})}\sim\frac{j}{\varepsilon}\,.\]

\textbf{Case 2: $n\neq m-2$}
\newline If $k=j=0$, we have
\[r^2R''+(m-n-1)R'=0\,,\]
whose solution is $R(r)=a+br^{2+n-m}$.
With the condition $R'(\delta)=0$, we obtain $b=0$.
With the condition $-R'(\varepsilon)=\sigma^N R(\varepsilon)$, we obtain
\[\sigma_{0,0}^N=0\,.\]

If $k=0$ and $j\neq 0$, we have
\[r^2R''+(m-n-1)rR'-\mu_{(j)}R=0\,,\]
whose solution is $R(r)=ar^j+br^{n+2-m-j}$.
With the condition $R'(\delta)=0$, we obtain $a=-b\frac{n+2-m-j}{j}\delta^{-2m}$. 
with the condition $-R'(\varepsilon)=\sigma^N R(\varepsilon)$, we obtain
\[\sigma_{0,(j)}^N=\frac{j((m-2-n+j+(n+2-m-j)\delta^{2+n-m-2j}\varepsilon^{m-n-2+2j})}{\varepsilon(j-(n+2-m-j)\delta^{2+n-m-2j}\varepsilon^{m-n-2+2j})}\sim\frac{m-2-n+j}{\varepsilon}\,.\]
Since the multiplicity of $\mu_{(j)}$ is $m_j$, the multiplicity of $\sigma^N_{0,(j)}$ is also $m_j$.

If $k\neq0$ and $j$ is arbitrary, we have
\[r^2R''+(m-n-1)rR'-(r^2\lambda_k+\mu_{(j)})R=0\,,\]
whose solution is $R(r)=r^\ell(aI_\nu(\sqrt{\lambda_k}r)+bK_\nu(\sqrt{\lambda_k}r))$, with $\ell=\frac{2+n-m}{2}$ and $\nu=\frac{m-n-2+2j}{2}$ (this solution is obtained like in the proof of Lemma \ref{lemme:asimptotiquedirichlet}).

With the condition $R'(\delta)=0$, we obtain $b=-a\mathcal{C}$, with
\[\mathcal{C}=\frac{\ell I_\nu(\sqrt{\lambda_k}\delta)+\sqrt{\lambda_k}\delta I'_\nu(\sqrt{\lambda_k}\delta)}{\ell K_\nu(\sqrt{\lambda_k}\delta)+\sqrt{\lambda_k}\delta K'_\nu(\sqrt{\lambda_k}\delta)}\,.\]
With the condition $-R'(\varepsilon)=\sigma^N R(\varepsilon)$, we obtain
\begin{gather*}
    \sigma^N =\frac{\ell\varepsilon^\ell(\mathcal{C}K_\nu(\sqrt{\lambda_k}\varepsilon)-I_\nu(\sqrt{\lambda_k}\varepsilon)+\varepsilon^\ell\sqrt{\lambda_k}(\mathcal{C}K'_\nu(\sqrt{\lambda_k}\varepsilon)-I'_\nu(\sqrt{\lambda_k}\varepsilon))}{\varepsilon^\ell(I_\nu(\sqrt{\lambda_k}\varepsilon)-\mathcal{C}K_\nu(\sqrt{\lambda_k}\varepsilon))}\\
    =-\frac{\ell}{\varepsilon}+\sqrt{\lambda_k}\frac{\mathcal{C}K'_\nu(\sqrt{\lambda_k}\varepsilon)-I'_\nu(\sqrt{\lambda_k}\varepsilon)}{I_\nu(\sqrt{\lambda_k}\varepsilon)-\mathcal{C}K_\nu(\sqrt{\lambda_k}\varepsilon)}\,.
\end{gather*}
With the recurrence relations \eqref{eq:deriveI}, \eqref{eq:deriveK} and the asymptotics \eqref{eq:asymptotiqueI}, \eqref{eq:asymptotiqueK}, 
\begin{gather*}
    \mathcal{C}\sim\frac{2j}{2+n-m-j}\Gamma(\nu)^{-1}\Gamma(\nu+1)^{-1}\bigg(\frac{1}{2}\sqrt{\lambda_k}\delta\bigg)^{2\nu}\,,\\
    \mathcal{C}K'_\nu(\sqrt{\lambda_k}\varepsilon)-I'_\nu(\sqrt{\lambda_k})\varepsilon\sim \frac{-\sqrt{\lambda_k}\bigg(\frac{1}{2}\sqrt{\lambda_k}\delta\bigg)^{2\nu}\bigg(\frac{1}{2}\sqrt{\lambda_k}\varepsilon\bigg)^{-\nu-1}}{2\Gamma(\nu)}\bigg(\frac{j}{2+n-m-j}-\bigg(\frac{1}{2}\sqrt{\lambda_k}\delta\bigg)^{-2\nu}\bigg(\frac{1}{2}\sqrt{\lambda_k}\varepsilon\bigg)^{2\nu}\bigg)\,,\\
    I_\nu(\sqrt{\lambda_k}\varepsilon)-\mathcal{C}K_\nu(\sqrt{\lambda_k}\varepsilon)\sim \frac{\bigg(\frac{1}{2}\sqrt{\lambda_k}\delta\bigg)^{2\nu}\bigg(\frac{1}{2}\sqrt{\lambda_k}\varepsilon\bigg)^{-\nu}}{\Gamma(\nu+1)}\bigg(\bigg(\frac{1}{2}\sqrt{\lambda_k}\delta\bigg)^{-2\nu}\bigg(\frac{1}{2}\sqrt{\lambda_k}\varepsilon\bigg)^{2\nu}-\frac{j}{2+n-m-j}\bigg)\,.
\end{gather*}
So,
\[\sigma_{k,(j)}^N\sim\frac{m-n-2+j}{\varepsilon}\,.\]
Since the multiplicity of $\mu_{(j)}$ is $m_j$, the multiplicity of $\sigma^N_{k,(j)}$ is also $m_j$.
\end{proof}

\section{Tubular excision of closed Riemannian manifold}\label{section:main}

We are now ready to prove Theorem \ref{thm:PrincipalIntro}. Let us recall that we need to show that, for all $k\,,\ell\geq 0$, except the case $k=\ell=0$,
\begin{gather*}
  \lim\limits_{\varepsilon\to0}\varepsilon\sigma_{k,\ell}(\Omega_\varepsilon)=m-n-2+j\,.
\end{gather*}
where $j\geq0$ is such that $j=\ell=0$ or $m_0+\cdots+m_{j-1}\leq\ell< m_0+\cdots+m_{j-1}+m_j$.
In particular, for $n=m-2$ and $\ell=0$, this limit is 0. In that case, the following improvement holds for each $k>0$,
\[\lim\limits_{\varepsilon\to0}\varepsilon|\log\varepsilon|\sigma_{k,0}(\Omega_\varepsilon)=1\,.\]

\begin{proof}[Proof of Theorem \ref{thm:PrincipalIntro}]
Let $\varepsilon_0>0$. By Proposition \ref{prop:quasi-iso}, there exists $\delta=\delta(\eps_0)>0$ such that $g$ and $\Tilde{g}$ are quasi-isometric with constant $1+\varepsilon_0$ on $\{p\in M~|~d_g(p,N)<\delta\}$. 
 
Let $0<\varepsilon<\delta$. Consider the Steklov problem on $\Omega_\varepsilon:=M\backslash\{p\in M~|~d_g(p,N)<\varepsilon\}$.
 
With the bracketing of Stekov eigenvalues with $A:=\{p\in M~|~\varepsilon<d_g(p,N)<\delta\}\subset M_\varepsilon$, we have that, for all $i\geq0$
\[\sigma_i^N(A,g)\leq\sigma_i(\Omega_\varepsilon,g)\leq\sigma_{i+1}^D(A,g)\,.\]
Since the spherical coordinates are well defined on $A_\varepsilon$, we can write $\Tilde{g}=h\oplus dr^2\oplus  r^2g_0$ and, by quasi-isometry,
\begin{gather*}
    \sigma_i^N(A,g)\geq\frac{\sigma_i^N((\varepsilon,\delta)\times N\times\s^{m-n-1},\Tilde{g})}{(1+\varepsilon_0)^{m+1/2}}\,,\\
    \sigma_{i+1}^D(A,g)\leq (1+\varepsilon_0)^{m+1/2}\sigma_{i+1}^D((\varepsilon,\delta)\times N\times\s^{m-n-1},\Tilde{g})\,.
\end{gather*}

Then,
\[\lim\limits_{\varepsilon\to0}\varepsilon\sigma_{k,\ell}(\Omega_\varepsilon,g)\leq\lim\limits_{\varepsilon\to0}\varepsilon\sigma_{k,\ell}^D(A,g)\leq\lim\limits_{\varepsilon\to0}\varepsilon(1+\varepsilon_0)^{2m+1}\sigma_{k,\ell}^D\,.\]
By Lemma \ref{lemme:asimptotiquedirichlet}, if $n=m-2$, we have

\begin{gather*}
    \varepsilon\sigma_{k,0}^D\sim\frac{1}{|\log\varepsilon|} \mbox{ for all $k\geq0$}\,,\\
    \varepsilon\sigma_{k,(j)}^D\sim j \mbox{ for all $j>0$ and for all $k\geq0$}\,.
\end{gather*}
If $n\neq m-2$, we have
\[\varepsilon\sigma_{k,(j)}^D\sim m-n-2+j \mbox{ for all $j\geq0$ and for all $k\geq0$}\,.\]
For every $j\geq0$, the multiplicity of $\sigma^D_{k,(j)}$ is $m_j$, the multiplicity of the distinct $j-$th Laplace eigenvalue of $\s^{m-n-1}$.

Thus, if $\ell=0$, set $j=0$ and if $\ell>0$, choose the unique $j>0$ such that $m_0+\cdots+m_{j-1}\leq\ell< m_0+\cdots+m_{j-1}+m_j$. We then have, for each $k\geq0$,
\[\lim\limits_{\varepsilon\to0}\varepsilon\sigma_{k,\ell}(\Omega_\varepsilon,g)\leq (1+\varepsilon_0)^{2m+1}m-n-2+j\,,\]
except where $k=\ell=0$.
Since it is true for every $\varepsilon_0>0$, we take the limit as $\varepsilon_0\to0$ to obtain
\[\lim\limits_{\varepsilon\to0}\varepsilon\sigma_{k,\ell}(\Omega_\varepsilon,g)\leq m-n-2+j\,.\]

We also have
\[\lim\limits_{\varepsilon\to0}\varepsilon\sigma_{k,\ell}(\Omega_\varepsilon,g)\geq\lim\limits_{\varepsilon\to0}\varepsilon\sigma_{k,\ell}^N(A,g)\geq\lim\limits_{\varepsilon\to0}\frac{\varepsilon}{(1+\varepsilon_0)^{2m+1}}\sigma_{k,\ell}^N\,.\]
By Lemma \ref{lemme:asymptotiqueneumann}, if $n=m-2$, we have
\begin{gather*}
    \varepsilon\sigma_{k,(j)}^N\sim j \mbox{ if $j\neq 0$ and $k\geq 0$}\,,\\
    \varepsilon\sigma_{k,0}^N\sim\frac{1}{|\log(\sqrt{\lambda_k}\varepsilon)|-\frac{K_0'(\sqrt{\lambda_k}\delta)}{I_0'(\sqrt{\lambda_k}\delta)}}\mbox{ for all $k>0$}\,.
\end{gather*}
If $n\neq m-2$, we have
\[\varepsilon\sigma_{k,(j)}^N\sim m-n-2+j\,,\]
for every $k\,,j\geq0$ except when $k=j=0$.
For every $j\geq0$, the multiplicity of $\sigma^N_{k,(j)}$ is $m_j$, the multiplicity of the distinct $j-$th Laplace eigenvalue of $\s^{m-n-1}$.

Thus, if $\ell=0$, set $j=0$ and if $\ell>0$, choose the unique $j>0$ such that $m_0+\cdots+m_{j-1}\leq\ell< m_0+\cdots+m_{j-1}+m_j$. We then have, for each $k\geq0$,
\[\lim\limits_{\varepsilon\to0}\varepsilon\sigma_{k,\ell}(\Omega_\varepsilon,g)\geq\frac{m-n-2+j}{(1+\varepsilon_0)^{2m+1}}\,,\]
except when $k=\ell=0$.

Since it is true for every $\varepsilon_0>0$, we take the limit as $\varepsilon_0\to0$ to obtain
\[\lim\limits_{\varepsilon\to0}\varepsilon\sigma_{k,\ell}(\Omega_\varepsilon,g)\geq m-n-2+j\,.\]

Thus, if $\ell=0$, set $j=0$ and if $\ell>0$, choose the unique $j>0$ such that $m_0+\cdots+m_{j-1}\leq\ell< m_0+\cdots+m_{j-1}+m_j$. Then, the following limit holds
\[\lim\limits_{\varepsilon\to0}\varepsilon\sigma_{k,\ell}(\Omega_\varepsilon,g)=m-n-2+j\,,\]
except when $k=\ell=0$.

When $\ell=0$, $k>0$ and $n=m-2$, we can improve the limit. Indeed, we have
\[\lim\limits_{\varepsilon\to0}\varepsilon|\log\varepsilon|\sigma_{k,0}(\Omega_\varepsilon,g)\leq\lim\limits_{\varepsilon\to0}\varepsilon|\log\varepsilon|\sigma_{k,0}^D(A,g)\leq \lim\limits_{\varepsilon\to0}\varepsilon|\log\varepsilon|(1+\varepsilon_0)^{2m+1}\sigma_{k,0}^D\sim (1+\varepsilon_0)^{2m+1}\,.\]
Since it is true for every $\varepsilon_0>0$, we take the limit as $\varepsilon_0\to0$ to obtain
\[\lim\limits_{\varepsilon\to0}\varepsilon|\log\varepsilon|\sigma_{k,0}(\Omega_\varepsilon,g)\leq 1\,.\]
Similarly, 
\[\lim\limits_{\varepsilon\to0}\varepsilon|\log\varepsilon|\sigma_{k,0}(\Omega_\varepsilon,g)\geq\lim\limits_{\varepsilon\to0}\varepsilon|\log\varepsilon|\sigma_{k,0}^N(A,g)\geq\lim\limits_{\varepsilon\to0}\frac{\varepsilon|\log\varepsilon|}{(1+\varepsilon_0)^{2m+1}}\sigma_{k,0}^N\sim\frac{1}{(1+\varepsilon_0)^{2m+1}}\,.\]
Since it is true for every $\varepsilon_0>0$, we take the limit as $\varepsilon_0\to0$ to obtain
\[\lim\limits_{\varepsilon\to0}\varepsilon|\log\varepsilon|\sigma_{k,0}(\Omega_\varepsilon,g)\geq 1\,.\]

So,
\[\lim\limits_{\varepsilon\to0}\varepsilon|\log(\varepsilon)|\sigma_{k,0}(\Omega_\varepsilon,g)=1\,.\]
\end{proof}

In the next example, we show that the behavior of the spectrum for submanifolds of dimension $m-1$ is different than for a submanifold of dimension $n\leq m-2$.

\begin{exemple}
Let $\mathbb{T}^2$ be the flat $2-$torus and $\gamma$ be the curve $(0,y)\sim(1,y)$. Consider the domain $\Omega_\eps:=\mathbb{T}^2\backslash\gamma_\eps$, where $\gamma_\eps$ is  a tubular neighbourhood of width $\eps$ around $\gamma$. The domain $\Omega_\eps$ is isometric to the cylinder $\s^1\times\lbrack\eps,1-\eps\rbrack$.
The Steklov problem on $\Omega_\eps$ is
$$
\begin{cases}
\partial_{ss}u+\partial_{tt}u=0 &\mbox{ in $\s^1\times(\eps,1-\eps)$}\,,\\
-u_t(s,\eps)=\sigma u(s,\eps) & \mbox{ on $\s^1\times\{\eps\}$}\,,\\
u_t(s,1-\eps)=\sigma u(s,1-\eps) &\mbox{ on $\s^1\times\{1-\eps\}$}\,.
\end{cases}
$$
Using seperation of variables, we find that the Steklov eigenvalues are
\[0\,,\frac{2}{1-2\eps}\,,k\coth\bigg(\frac{k(1-2\eps)}{2}\bigg)\,,k\tanh\bigg(\frac{k(1-2\eps)}{2}\bigg)\,.\]
Taking the limit as $\eps\to0$, we obtain
\[0\,,2\,,k\coth\bigg(\frac{k}{2}\bigg)\,,k\tanh\bigg(\frac{k}{2}\bigg)\,.\]

\end{exemple}

Instead of proving Theorem \ref{thm:pointIntro}, let us prove a slightly different but equivalent result:
\begin{theoreme}\label{thm:point}
Let $M$ be a smooth compact Riemannian manifold of dimension $m\geq 2$ and $p\in M$. Then, if $j=0$, set $k=0$ and for $j>0$, choose the unique $k>0$ such $m_0+\cdots+m_{k-1}\leq j< m_0+\cdots+m_{k-1}+m_k$. The following limit holds 
\[\lim\limits_{\varepsilon\to0}\sigma_k(\Omega_\eps)|\partial\Omega_\eps|^{1/(m-1)}=(m+k-2)\omega_{m-1}^{1/(m-1)}\,,\]
where $\Omega_\eps:=M\setminus B(p,\eps)$ and $\omega_{m-1}=|\s^{m-1}|$.

\end{theoreme}
\begin{proof}
Let $\varepsilon_0>0$. By Remark \ref{remarque:quasi-isopoint}, there exists $\delta>0$ such that $\Tilde{g}$ and $g$ are quasi-isometric with constant $1+\varepsilon_0$ on $B_\delta(p)$.

Let $0<\varepsilon<\delta$ and consider the Steklov problem on $\Omega_\varepsilon=M\backslash B_\delta(p)$.
By the bracketing of Steklov eigenvalues with $A:=B_\delta(p)\backslash B_\varepsilon(p)$, we have
\[\sigma_i^N(A,g)\leq\sigma_i(\Omega_\varepsilon,g)\leq\sigma_{i+1}^D(A,g)\,.\]
Since the spherical coordinates are well defined on $A$, we can write $\Tilde{g}=dr^2\oplus r^2g_0$ and, by quasi-isometry,
\begin{gather*}
    \sigma_i^N(A,g)\geq\frac{\sigma_i^N(A,\Tilde{g})}{(1+\varepsilon_0)^{m+1/2}}\,,\\
    \sigma_{i+1}^D(A,g)\leq\sigma_{i+1}^D(A,\Tilde{g})(1+\varepsilon_0)^{m+1/2}\,,
\end{gather*}
where $\sigma_i^N(A,\Tilde{g})$ and $\sigma_{i+1}^D(A,\Tilde{g})$ are the eigenvalues of the problems
\[\begin{cases}
u_{rr}+(m-1)r^{-1}u_r+r^{-2}\Delta_{\s^{m-1}}u=0 &\mbox{ in $A$}\,,\\
\partial_nu=0&\mbox{ on $\partial B_\delta(p)$}\,,\\
-\partial_n u=\sigma^Nu &\mbox{ on $\partial B_\varepsilon(p)$}\,,
\end{cases}\]
\[\begin{cases}
u_{rr}+(m-1)r^{-1}u_r+r^{-2}\Delta_{\s^{m-1}}u=0 &\mbox{ in $A$}\,,\\
u=0&\mbox{ on $\partial B_\delta(p)$}\,,\\
-\partial_n u=\sigma^Du &\mbox{ on $\partial B_\varepsilon(p)$}\,.
\end{cases}\]
By quasi-isometry, we also have
\[\frac{\varepsilon\omega_{m-1}^{1/(m-1)}}{(1+\varepsilon_0)^{(m-1)/2}}\leq |\partial \Omega_\varepsilon|^{1/(m-1)}\leq (1+\varepsilon_0)^{(m-1)/2}\varepsilon\omega_{m-1}^{1/(m-1)}\,.\]

By separation of variables, let us find the harmonic functions on $\Omega_\varepsilon$. Suppose $u(r,p)=F(r)G(p)$. Then
\[u_{rr}+(m-1)r^{-1}u_r+r^{-2}\Delta_{\s^{m-1}}u=0\]
implies that
\[\frac{r^2F''+r(m-1)F'}{F}=\frac{-\Delta_{\s^{m-1}}G}{G}=\lambda\,.\]
The equation $-\Delta_{\s^{m-1}}G=\lambda G$ gives us the solutions $\lambda_{(k)}=k(k+m-2)$, with the associated eigenfunction $G_k$ which is a spherical harmonic of degree $k$. 
The multiplicity of $\lambda_{(k)}$ is also $m_k$.

Then, we solve $r^2F''+r(m-1)F'-k(k+m-2)F=0$ for all $k\geq0$.
We find the different solutions:
\begin{gather*}
    a+b\log r \mbox{ , when $k=0$ and $m=2$}\,,\\
    a+br^{2-m} \mbox{ , when $k=0$ and $m\neq2$}\,,\\
    ar^k+br^{2-m-k} \mbox{ , otherwise}\,.
\end{gather*}
For the Steklov-Neumann problem, we find the following eigenvalues for $k\geq1$:
\[\sigma_k^N(A,\Tilde{g})=\frac{k(m+k-2)(1-\delta^{2-m-2k}\varepsilon^{m+2k-2})}{\varepsilon(k+(m+k-2)\delta^{2-m-2k}\varepsilon^{m+2k-2})}\,.\]
For the Steklov-Dirichlet problem, we find the following eigenvalues for $k\geq1$:
\[\sigma_{k+1}^D(A,\Tilde{g})=\frac{(m+k-2)+k\delta^{2-m-2k}\varepsilon^{m+2k-2}}{\varepsilon(1-\delta^{2-m-2k}\varepsilon^{m+2k-2})}\,.\]
The multiplicity of these eigenvalues is also $m_k$.

Thus, if $j=0$, set $k=0$ and for $j>0$, choose the unique $k>0$ such that $m_0+\cdots+m_{k-1}\leq j< m_0+\cdots+m_{k-1}+m_k$. In this case, we have
\begin{equation*}
    \frac{(m+k-2)\omega_{m-1}^{1/(m-1)}}{(1+\varepsilon_0)^{(2m^2-1)/(2m-2)}}\leq\lim\limits_{\varepsilon\to0}\sigma_j(\Omega_\varepsilon)|\partial \Omega_\varepsilon|^{1/(m-1)}\\
    \leq(1+\varepsilon_0)^{(2m^2-1)/(2m-2)}(m+k-2)\omega_{m-1}^{1/(m-1)} \,.
\end{equation*}
Since it is true for every $\varepsilon_0$, we take the limit as $\varepsilon_0\to0$. This concludes the proof.

\end{proof}

\section{Acknowledgements}\label{section:ack}

The author would like to thank Bruno Colbois and Jean Lagacé for reading an early version of the article and Léonard Tschanz and Bruno Colbois for helping with the multiplicity in the proof of Theorem \ref{thm:pointIntro}.
The author is supported by NSERC. This work is a part of the PhD thesis of the author under the supervision of Alexandre Girouard.

\bibliographystyle{plain}
\bibliography{bibliography}

\end{document}